\documentclass[11pt]{amsart}

\newtheorem{theorem}{Theorem}[section]
\newtheorem{proposition}[theorem]{Proposition}

\newtheorem{lemma}[theorem]{Lemma}
\newtheorem{claim}[theorem]{Claim}
\newtheorem{fact}[theorem]{Fact}
\newtheorem{cor}[theorem]{Corollary}
 
\newtheorem{definition}[theorem]{Definition}

\newtheorem{conjecture}[theorem]{Conjecture}

\theoremstyle{plain}
\numberwithin{equation}{theorem}

\theoremstyle{remark}

\newcommand{\C}{{\mathbb C}}

\newcommand{\Q}{{\mathbb Q}}

\newcommand{\Qbar}{\bar{\Q}}

\DeclareMathOperator{\bN}{\mathbb{N}}

\newcommand{\bG}{{\mathbb G}}

\newcommand{\bC}{{\mathbb C}}

\newcommand{\bA}{{\mathbb A}}
\newcommand{\bQ}{{\mathbb Q}}

\newcommand{\lra}{\longrightarrow}

\newcommand{\cO}{\mathcal{O}}

\newcommand{\cS}{\mathcal{S}}
\newcommand{\cP}{\mathcal{P}}

\title{$p$-adic logarithms for polynomial dynamics}
\author{D.~Ghioca and T.~J.~Tucker}
\thanks{The second author was partially supported by National Security
  Agency Grant 06G-067}

\address{
Dragos Ghioca \\
Department of Mathematics \& Computer Science\\
University of Lethbridge \\
4401 University Drive \\ 
Lethbridge, Alberta T1K 3M4 
}

\email{ghioca@cs.uleth.ca}

\address{
Thomas Tucker\\
Department of Mathematics\\
Hylan Building\\
University of Rochester\\
Rochester, NY 14627
}

\email{ttucker@math.rochester.edu}

\begin{document}

\begin{abstract}
  We prove a dynamical version of the Mordell-Lang conjecture for
  subvarieties of the affine space $\bA^g$ over a $p$-adic field,
  endowed with polynomial actions on each coordinate of $\bA^g$. We
  use analytic methods similar to the ones employed by Skolem,
  Chabauty, and Coleman for studying diophantine equations.
\end{abstract}

\maketitle

\section{Introduction}
\label{intro}

The Mordell-Lang conjecture was proved by Faltings \cite{Faltings}.
\begin{theorem}[Faltings]
\label{T:F}
Let $G$ be an abelian variety defined over the field of complex
numbers $\mathbb{C}$. Let $X\subset G$ be a closed subvariety and
$\Gamma\subset G(\mathbb{C})$ a finitely generated subgroup of
$G(\C)$. Then $X(\mathbb{C})\cap\Gamma$ is a finite union of cosets
of subgroups of $\Gamma$.
\end{theorem}

In particular, Theorem~\ref{T:F} says that if a subvariety $X$ of an
abelian variety $G$ does not contain a translate of a positive
dimension algebraic subgroup of $G$, then $X$ has a finite
intersection with any finitely generated subgroup of $G(\C)$.
Theorem~\ref{T:F} has been generalized to semiabelian varieties $G$ by
Vojta (see \cite{V1}) and to finite rank subgroups $\Gamma$ of $G$ by
McQuillan (see \cite{McQ}). Recall that a semiabelian variety (over
$\mathbb{C}$) is an extension of an abelian variety by a torus
$(\mathbb{G}_m)^k$. The authors have proved the following dynamical version
of Theorem~\ref{T:F} (see \cite{endomorphism}): if $\phi$ is an
endomorphism of a semiabelian variety $G$ defined over $\C$, and
$V\subset G$ has no positive dimensional subvariety invariant under
$\phi$, then any orbit of $\phi$ has finite intersection with $V$.  In
case $G=\mathbb{G}_m^k$, this says that if an affine variety
$V\subset\mathbb{G}_m^k$ contains no subvariety which is invariant
under the map $(X_1,\dots,X_k)\mapsto (X_1^{e_1},\dots,X_k^{e_k})$
(with $e_i\in\mathbb{N}$), then $V$ has finite intersection
with the orbit of any point of $\bA^k$ under the above map on $G$.  It
is natural to ask whether a similar conclusion holds for any
polynomial action on $\bA^k$.  In this spirit, we propose the
following conjecture.

\begin{conjecture}
\label{dynamical M-L}
Let $f_1,\dots,f_g$ be polynomials in $\bC[X]$, let $\cP$ be their
action coordinatewise on $\bA^g$, and let $V$ be a subvariety of
$\bA^g$ that does not contain a positive dimensional subvariety
periodic under the $\mathcal{P}$-action.  Then the $\cP$-orbit of any
point in $\bA^g$ intersects $V$ in at most finitely many points.
\end{conjecture}

We also propose the following more general conjecture for the
structure of the intersection of an affine variety with a polynomial
orbit. We let $\cO_{\cP}((a_1,\dots,a_g))$ denote the $\cP$-orbit of
$(a_1,\dots,a_g)\in \bA^g(\bC)$.
\begin{conjecture}
\label{dynamical M-L 2}
With the above notation, each subvariety $V$ of $\bA^g$ defined over $\bC$
intersects $\cO_{\cP}((a_1,\dots,a_g))$ in a finite union of orbits of the form $\cO_{\cP^N}(\cP^{\ell}(a_1,\dots,a_g))$, for some non-negative integers $N$ and $\ell$.
\end{conjecture}

Conjecture~\ref{dynamical M-L} is an easy corollary of
Conjecture~\ref{dynamical M-L 2}. Indeed, if the intersection
$V(\C)\cap\cO_{\cP}((a_1,\dots,a_g))$ is infinite, then there exists
an infinite orbit $\cO_{\cP^N}(\cP^{\ell}(a_1,\dots,a_g))$ contained
in $V$, and the Zariski closure of
$\cO_{\cP^N}(\cP^{\ell}(a_1,\dots,a_g))$ contains a positive
dimensional subvariety of $V$ invariant under $\cP^N$. Note that
Conjecture~\ref{dynamical M-L 2} says that if $S$ is the set of
non-negative integers $n$ for which $\cP^n(a_1, \dots, a_g)$ lies in $V$,
then $S$ is a finite union of translates of semigroups of $\mathbb{N}$.

Conjectures~\ref{dynamical M-L} and \ref{dynamical M-L 2} fit into Zhang's far-reaching system of dynamical
conjectures \cite{ZhangLec}.  Zhang's conjectures include dynamical
analogues of the Manin-Mumford and Bogomolov conjectures for abelian
varieties (now theorems of Raynaud \cite{Raynaud1, Raynaud2}, Ullmo \cite{Ullmo},
and Zhang \cite{Zhang}), as well as a conjecture about the Zariski
density of orbits of points under fairly general maps from a
projective variety to itself.  The latter conjecture is related to our
Conjecture~\ref{dynamical M-L}, though neither conjecture contains the
other. 

Conjecture~\ref{dynamical M-L 2} is proved in \cite{endomorphism} when each $f_i\in\Qbar[X]$, and $\deg(f_i) \le 1$. We also note that in \cite{Mike},
Conjecture~\ref{dynamical M-L} is proved in the special case where
$g=2$ and $V$ is a line in $\bA^2$.

An analog of our Conjecture~\ref{dynamical M-L 2} for the additive group
of positive characteristic under the action of an additive polynomial
associated to a Drinfeld module has been previously studied (see
\cite{Denis-conjectures}, \cite{IMRN}, \cite{full-ml-drinfeld} and
\cite{compositio}). However, Conjecture~\ref{dynamical M-L} seems more
difficult, since the proofs over Drinfeld modules make use of the fact
that the polynomials there are additive; in particular, this allows
for the definition of a logarithm-like map that is defined on all of
$\bG_a$.

In the present paper we prove a first result towards
Conjecture~\ref{dynamical M-L}, one that is valid for polynomials
defined over the $p$-adics. The idea behind the proof of our
Theorem~\ref{polydyn} can be explained quite simply. Assuming that an
affine variety $V\subset\bG_a^g$ has infinitely many points in common
with an orbit $\cO$ of a point which lies in a sufficiently small
neighborhood of an attracting fixed point for $\cP$, we can find then
a positive dimensional subvariety of $V$ which is invariant under
$\cP$. Indeed, applying the $p$-adic logarithmic map associated to
$\cP$ (see Proposition~\ref{expo}) to $\cO$ yields a line in the
vector space $\bC_p^g$.  Each polynomial $f$ that vanishes on $V$,
then gives rise to an analytic function $F$ on this line (by composing
with the $p$-adic exponential function associated to $\cP$).  Because
we assumed there are infinitely many points in $V \cap \cO$, the zeros
of $F$ must have an accumulation point on this line, which means that
$F$ vanishes identically on the line (by Lemma~\ref{zeros}).  The
Zariski closure of this analytic line is a subvariety of $V$ which is
$\cP$-invariant.  The inspiration for this idea comes from the method
employed by Chabauty in \cite{Chabauty} (and later refined by Coleman
in \cite{Coleman}) to study rational points on curves in abelian
varieties with low Mordell-Weil rank.  Our technique also bears a
resemblance to Skolem's method for treating diophantine equations (see
\cite[Chapter 4.6]{BS}).

We briefly sketch the plan of our paper. In Section~\ref{notation} we
set up the notation and state our main result (Theorem~\ref{polydyn}).
Section~\ref{p-adic} is devoted to proving several lemmas for $p$-adic
logarithms associated to (one variable) polynomial maps on $\bC_p$;
these lemmas are used in the proof of Theorem~\ref{polydyn}.  After
that, in Section~\ref{proofs} we complete the proof of
Theorem~\ref{polydyn}.

\section{Notation}
\label{notation}

Let $p$ be a prime number and let $\bC_p$ be the completion of an
algebraic closure of the field of $p$-adic numbers $\bQ_p$. We let $|\cdot |$ be the absolute value on $\bC_p$. For
$\alpha\in\bC_p $ and a real number $R>0$, we let
$$B(\alpha;R):=\{z\in\bC_p\text{ : } |z-\alpha|  <R\}.$$
An \emph{isometry} between $B(\alpha_1;R)$ and $B(\alpha_2;R)$ is a map 
$$\psi:B(\alpha_1;R)\lra B(\alpha_2;R)$$ 
such that for each $z\in B(\alpha_1;R)$, we have
$$|\psi(z)-\alpha_2| = |z-\alpha_1|.$$

For every polynomial $P\in \bC_p[X]$, and for every $n\in\bN$, we let
$P^n$ denote the $n$-th iterate of $P$; that is, we let
$$P^n:=P\circ P\circ \dots \circ P\text{ ($n$ times).}$$

We call $\alpha\in\bC_p$ a \emph{preperiodic} point for $P$ if there
exist non-negative integers $n \not= m$ such that $P^n(\alpha) =
P^m(\alpha)$.

Let $g\ge 2$, and let $P_1,\dots,P_g \in \bC_p[X]$. We denote as $\cP$
the action of $(P_1,\dots,P_g)$ on $\bA^g$ coordinatewise. For each
$(x_1,\dots,x_g)\in \bC_p^g$, we define the \emph{($\cP$-) orbit} of
$(x_1,\dots,x_g)$ be
$$\cO_{\cP}((x_1,\dots,x_g)):=\{\left(P^n_1(x_1),\dots,P^n_g(x_g)\right)  \text{ : } n\ge 0\}.$$ 

We will use the following classical definition from complex dynamics.
\begin{definition}
\label{definition attracting fixed point}
Let $P$ be a polynomial in $\bC_p[z]$.  We call $\alpha\in \bC_p$ a
{\bf fixed point} of $P$ if $P(\alpha) = \alpha$.  We say that a fixed
point $\alpha$ is an {\bf attracting} fixed point for $P$ if $0 <
|P'(\alpha)| < 1$. If $P'(\alpha) = 0$ for a fixed point $\alpha$,
then we call $\alpha$ a {\bf superattracting} fixed point. We call
$P'(\alpha)$ the {\bf multiplier} of $\alpha$.
\end{definition}

The following is our main result.

\begin{theorem}
\label{polydyn}
With the above notation for $\cP$, let $(\alpha_1,\dots,\alpha_g)\in
\bC_p^g$, and assume for each $i$ that $\alpha_i$ is an attracting
(but not superattracting) fixed point for $P_i$. In addition, suppose
that the $\alpha_i$ have the same multiplier $a_1$.  Let $V$ be an affine subvariety of $\mathbb{A}^g$
defined over $\bC_p$. Assume that $V$ does not contain a positive
dimensional subvariety invariant under the $\cP$-action. Then there
exists $R>0$ such that for any $(x_1,\dots,x_g)\in\bC_p^g$ that
satisfies $\max_{i=1}^g|x_i-\alpha_i| < R$, the intersection $V(\bC_p)\cap \cO_{\cP}((x_1,\dots,x_g))$ is finite.
\end{theorem}

\section{$p$-adic logarithms}
\label{p-adic}

We will prove Theorem~\ref{polydyn} after we develop a theory of
$p$-adic logarithms associated to polynomials defined over $\bC_p$.
We begin with a variant of the classical K\"{o}nigs linearization of a
polynomial $P$ at an attracting (but not superattracting) fixed point
$\alpha$ (see Theorem $8.2$ in \cite{Milnor}, or Theorem $2.1$ in
\cite[Chapter 2]{Carleson-Gamelin}). In both of those books, the
result is proved over the complex numbers, and under the hypothesis
that $0 < |P'(\alpha)| < 1$.  By contrast, our result is over the
$p$-adics and is less restrictive in that we only require $P'(\alpha)$
be neither $0$ nor a root of unity. Our proof is also different than
the proofs from the above mentioned books.

We note that we get a {\it convergent} power series for $\exp_P$ as long
as $P'(\alpha)$ is neither $0$, nor a root of unity.  The reason roots of unity are a problem is that we divide
out by $P'(\alpha)^n - 1$ for various $n$ when we are solving for the
coefficients of $\exp_P$.  In order to control the size of the
coefficients, we will need a lemma on the size of $|P'(\alpha)^n - 1|$.
We begin with the following.

\begin{lemma}\label{schin}
  Let $\beta \in \bC_p$ have the property that $|\beta - 1|  <
  |p| $.  Then for all positive integers
  $n$, we have
$$ |\beta^n - 1|  = |\beta - 1|  \cdot |n| .$$
\end{lemma}
\begin{proof}
We use induction on the maximal power of $p$ that divides $n$.  If
this power is zero, then we have
\begin{equation*}
\begin{split}
|\beta^n - 1|  & = |\beta - 1| \cdot |1 + \beta + \dots + \beta^{n-1}|  \\
& = |\beta - 1| \cdot \big| 0 + (\beta - 1) + \dots + (\beta^{n-1} -
  1) + n \big|  \\
& = |\beta - 1| \cdot |n| ,
\end{split}
\end{equation*}
since $|\beta^j - 1|  < |n| =1$ for all $j$.  To perform the
inductive step, we let $n' = n/p$ and similarly obtain $|\beta^n -
1|  = |\beta^{n'} - 1| \cdot |p| $, since $|(\beta^{n'})^j - 1|  <
|p| $ for all $j$.  Then, by the inductive hypothesis, we have
$$|\beta^n - 1|  = |(\beta^{n'})^p - 1| =  |\beta-1| |n'|\cdot |p|  = |\beta-1| \cdot |n| ,$$
as desired.
\end{proof} 

The following Lemma is an immediate consequence of Lemma~\ref{schin}.
Schinzel (\cite[Lemma 3]{Schinzel}) uses a similar lemma but only
proves it over number fields.
\begin{lemma}
\label{basic_fact_schinzel}
Assume $b\in \bC_p\setminus\{0\}$ is not a root of unity. Then there
exists $0<C\le 1$ such that for every $n>1$, we have $|b^n-b|  \ge
C\cdot |n-1| $.
\end{lemma}

We are now ready to prove the existence of our exponential map when
the fixed point $\alpha = 0$.

\begin{proposition}
\label{expo}
Let $P(X) =\sum_{i=1}^r a_iX^i \in \bC_p[X]$, where $a_1\ne 0$ is not
a root of unity.  Then there exists a power series $$\exp_{P}(X) = X +
c_2X^2 +\dots + c_nX^n +\dots \in \bC_p[[X]]$$
such that
$$P(\exp_P(X)) = \exp_P(a_1X)$$
and $\exp_P$ has a positive radius of
convergence.
\end{proposition}

In particular, Proposition~\ref{expo} shows that for every $n\ge 1$, we have
\begin{equation}
\label{fundamental relation 1}
P^n(\exp_P(X)) = \exp_P(a_1^nX),
\end{equation}
as an identity of formal power series.
Also, we obtain that there exists a \emph{logarithmic} function
$\log_P$, which equals the inverse $\exp_{P}^{-1}$ of the
\emph{exponential} function from Proposition~\ref{expo}. Moreover, since the
radius of convergence for $\exp_P$ is positive, and because the
coefficient of the linear term of $\exp_P$ is equal to $1$, there
exists $r_0>0$ such that both $\exp_P$ and $\log_P$ are analytic
isometries on $B(0;r_0)\subset\bC_p$ (see Proposition~\ref{non-arch
  anal}).  Finally, using \eqref{fundamental relation 1}, we also
derive
\begin{equation}
\label{fundamental relation 2}
\log_P(P^n(X)) = a_1^n \log_P(X).
\end{equation}

\begin{proof}[Proof of Proposition~\ref{expo}.]
  We let $c_1=1$ and solve inductively for $c_n$ with $n \ge 2$ by
  equating the coefficient of $X^n$ in $P(\exp_P(X))$ with the
  coefficient of $X^n$ in $\exp_P(a_1X)$ (it is clear that the
  coefficient of $X$ in both power series must be equal to $a_1$). The
  coefficient of $X^n$ in $ P(\exp_P(X))$ must equal
  $$a_1 c_n + \sum_{i=2}^r a_i\cdot\sum_{\substack{j_1+\dots +j_i =n\\
      j_1,\dots,j_i\ge 1}}\left(\prod_{\ell=1}^i
    c_{j_{\ell}}\right).$$
  The coefficient of $X^n$ in $\exp_P(a_1X)$
  must equal $c_n a_1^n$. Thus, we may solve for $c_n$ by letting
\begin{equation}
\label{recursive relation for coefficients}
(a_1^n-a_1)c_n = \sum_{i=2}^r a_i\cdot\sum_{\substack{j_1+\dots +j_i =n\\ j_1,\dots,j_i\ge 1}}\left(\prod_{\ell=1}^i c_{j_{\ell}}\right).
\end{equation}
Note that since $a_1$ is neither $0$ nor a root of unity, we have
$a_1^n-a_1\ne 0$, so this equation does indeed have a solution for
$c_n$ in terms of $a_1$ and the $c_j$ for which $j < n$. This shows
the existence of the formal power series for the exponential function
$\exp_P$.

We now prove that this formal power series has a positive radius of
convergence in $\bC_p$. Let $M:=\min\{1,\frac{1}{\max_{i=2}^r
  |a_i| }\}$.  Since $a_1$ is neither a root of unity nor equal to
$0$, Lemma~\ref{basic_fact_schinzel} implies that there exists $0<C\le
1$ such that for each $n\ge 2$, we have $|a_1^n-a_1|  \ge C|n-1| $.
We show by induction on $n$ that $|c_n|  \le (C\cdot M)^{1-n}\cdot
|(n-1)!|^{-1} $ for each $n\ge 1$.

For $n=1$, we know that $c_1=1$, so the desired inequality holds.

Now, let $n\ge 2$. Assume that $|c_i| \le (C\cdot
M)^{1-i}|(i-1)!|^{-1} $ for each $i<n$; we will show that $|c_n| \le
(C\cdot M)^{1-n}|(n-1)!|^{-1} $.  Using \eqref{recursive relation for
  coefficients} and the ultrametric inequality, we conclude that
\begin{equation}
C|n-1| \cdot |c_n|  \le |a_1^n-a_1|  \cdot |c_n|  \le
\big(\max_{i=2}^r |a_i|  \big) \cdot \big(
  \max_{\substack{j_1+\dots +j_i =n\\ j_1,\dots,j_i\ge
      1}}\prod_{\ell=1}^i \left|c_{j_{\ell}}\right|  \big).
\end{equation}
Since $|a_i| \le \frac{1}{M}$, the induction hypothesis (applied to
the various $c_{j_{\ell}}$) above thus yields
\begin{equation}
\label{second equation}
C|n-1| \cdot |c_n|  \le \left(\frac{1}{M} \right)\cdot \left(\max_{i=2}^r (C\cdot
M)^{i-n}\right) \cdot \left( \max_{\substack{j_1+\dots+j_i=n \\ i\ge 2}}\prod_{\ell=1}^i
  \left|(j_{\ell}-1)!\right|^{-1}  \right).
\end{equation}
Both $C$ and $M$ are in $(0,1]$, so $\max_{i=2}^r (C\cdot
M)^{i-n}=(C\cdot M)^{2-n}$. For each $i,j_1,\dots,j_i$ such that
$j_1+\dots+j_i=n$, we have
$$\frac{(n-i)!}{(j_1-1)!\cdots (j_i-1)!}\in\mathbb{Z}.$$
This gives
$\prod_{\ell=1}^i |(j_i-1)!|^{-1}  \le |(n-i)!|^{-1} \le
|(n-2)!|^{-1} $ and, thus, \eqref{second equation} implies that
\begin{equation}
\label{third equation}
C|n-1| \cdot |c_n| \le C^{2-n}\cdot M^{1-n}\cdot |(n-2)!|^{-1} ,
\end{equation}
which completes the inductive step.

Since $|(n-1)!|^{-1}  \le |p| ^{-\frac{n-1}{p-1}} \le |p|^{-n}$, we conclude that $|c_n|  \le \left(C M|p|\right)^{-n}$. Hence $\exp_P$ is convergent in the ball $B\left(0;CM|p|\right)$.
\end{proof}

The following result is classical in non-archimedean analysis. For the sake of completeness, we provide its proof.
\begin{proposition}
\label{non-arch anal}
Let $f(X)= X + c_2 X^2 + \dots + c_n X^n + \dots\in\bC_p [[X]]$ be a power series convergent on a ball $B(0;r)$ of positive radius. Then there exists $r_0\in (0,r]$ such that for each $z\in B(0;r_0)$, we have $|f(z)|  = |z| $. Moreover, $f$ admits an analytic inverse function $f^{-1}$ on $B(0;r_0)$.
\end{proposition}

\begin{proof}
Let $0<r_1<r$. Then $\limsup_{n\to\infty} |c_n| ^{\frac{1}{n}} < \frac{1}{r_1}$. Therefore, there exists $N_1\in\mathbb{N}$ such that for every $n\ge N_1$, we have $|c_n| ^{\frac{1}{n}} < \frac{1}{r_1}$. In particular, there exists $K>0$ such that for \emph{all} $n\ge 2$, we have
\begin{equation}
\label{limsup 1}
|c_n|  < \frac{K}{r_1^n}.
\end{equation}
We can find $0<r_0<r_1$ such that for all $n\ge 2$, we have
\begin{equation}
\label{limsup 2}
r_0^{n-1} < \frac{r_1^n}{K}.
\end{equation}
We claim that for each $z\in B(0;r_0)$, we have $|f(z)| =|z| $. Indeed, for each $n\ge 2$, using \eqref{limsup 1}, \eqref{limsup 2} and that $|z|  < r_0$, we have
$$|c_nz^n|  < \frac{|z|^n\cdot K}{r_1^n} < |z| .$$
Hence $|f(z)| =|z| $, as desired. Moreover, for each $w \in B(0;r_0)$,
there exists $z\in B(0;r_0)$ such that $f(z) = w$ (see the first slope
of the Newton polygon for the polynomial $f(X) - w$). Thus
$f:B(0;r_0)\to B(0;r_0)$ is a bijection.

Since $f$ is a unit in the ring of formal power series, there exists a
formal power series $f^{-1}$ which is the inverse of $f$.
Furthermore, the fact that $f$ is one-to-one and onto $B(0;r_0)$ means
that $f^{-1}$ is well-defined and analytic on $B(0;r_0)$ (because $f$
is analytic on $B(0;r_0)$). Furthermore, $|f^{-1}(z)| = |z| $, for
each $z\in B(0;r_0)$ since $|f(z)| = |z|$ for each $z\in B(0;r_0)$.
\end{proof}

The following result is a nonarchimedean version of a classical result
in dynamics.  It follows in a similar manner as Proposition~\ref{non-arch anal} from considering $P(X)$ as a
polynomial in $(X - \alpha)$ and applying the ultrametric inequality
for nonarchimedean absolute values.
\begin{fact}
\label{dynamics}
Let $P\in \bC_p[X]$ and let $\alpha\in \bC_p$ be a fixed point for
$P$. Assume that $0<|P'(\alpha)|<1$ (i.e. $\alpha$ is an attracting
fixed point for $P$).  Then there exists $R > 0$ such that for each
$z\in B(\alpha;R)\setminus\{0\}$, we have
\begin{equation}\label{from fact}
|P(z)-\alpha|  = |P'(\alpha)| \cdot |z-\alpha| < |z-\alpha| .
\end{equation}  
Thus, there exists $R>0$ such that for every
$z\in B(\alpha;R)$, we have 
$$\lim_{n\to\infty}P^n(z)=\alpha.$$
\end{fact}

Using Proposition~\ref{expo}, we construct a $p$-adic logarithmic function in a
neighborhood of any attracting fixed point of a polynomial $P$.

\begin{proposition}
\label{extension-trick}
Let $P(X)\in \bC_p[X]$, and suppose that $\alpha \in  \bC_p$ is a fixed point of $P$. Assume that $0<|P'(\alpha)| < 1$. Then there exists $R>0$, and there exist convergent power series $\exp_P:B(0;R)\to B(\alpha;R)$ and $\log_P:B(\alpha;R)\to B(0;R)$ (which are inverses of each other), such that
$P(\exp_P(X)) = \exp_P(b_1X)$ and $\log_P(P(X))=b_1\log_P(X)$,
where $b_1:=P'(\alpha)$.
\end{proposition}

\begin{proof}
  The proof is a simple change of coordinates argument followed by the
  application of Proposition~\ref{expo}.  Let $G(X):= P(X+\alpha) -
  \alpha$.  Then $G(0) = 0$ and $G'(0)=P'(\alpha)=b_1$. 
Note that $b_1$ is neither $0$ nor a root of unity.
By Propositions~\ref{expo} and \ref{non-arch anal}, we have analytic
exponential and logarithmic functions for $G$ in a neighborhood of
$0$; we denote these as $\exp_G$ and $\log_G$.  Recall that $\log_{G}$
and $\exp_{G}$ are isometries on $B(0;R)$ for sufficiently small $R$. At the expense of shrinking $R$, we may assume that $P$ maps $B(\alpha;R)$ into itself (see Fact~\ref{dynamics}), not necessarily onto.
Then $\exp_P(X):=\exp_G(X) + \alpha$ and $\log_P(X):=\log_G(X-\alpha)$ are inverses of each other, and for any $X\in B(\alpha;R)$ (see Proposition~\ref{expo}), we have
\begin{equation}
\label{fundamental relation 3}
\begin{split}
  \log_P(P(X)) & = \log_{G}(P (X)-\alpha)\\
 & = \log_{G}( G(X-\alpha))\\ & = b_1
  \log_{G}(X-\alpha) \\
& = b_1\log_P(X),
\end{split}
\end{equation}
and similarly, we have $P(\exp_P(X)) = \exp_P(b_1 X)$ for $X\in B(0;R)$.  
\end{proof}


\section{Proof of our main result}
\label{proofs}

We are now almost ready to prove Theorem \ref{polydyn}, which is the main result
of this paper.  The proof makes use of the fact that the zeros of any
analytic function are isolated, unless the function is identically
zero.  The following lemma is standard (see \cite[Lemma
3.4]{compositio}, for example).  

\begin{lemma}\label{zeros}
  Let $F(z) = \sum_{i=0}^{\infty} a_i (z-\alpha)^i$ be a power series with
  coefficients in $\bC_p $ that is convergent in an open disc $B$ of
  positive radius around the point $z = \alpha$.  Suppose that $F$ is not
  the zero function.  Then the zeros of $F$ in $B$ are isolated.
\end{lemma}

We now begin the proof of Theorem \ref{polydyn}.  

\begin{proof}
  Let $R>0$ satisfy the hypothesis of Fact~\ref{dynamics} for each
  $P_i$. Let $x_i\in B(\alpha_i;R)$ for each $i\in\{1,\dots,g\}$. If each $x_i=\alpha_i$, then $\cO_{\cP}((x_1,\dots,x_g))=(\alpha_1,\dots,\alpha_g)$, and so, the conclusion of Theorem~\ref{polydyn} is immediate. Hence, we may assume that 
\begin{equation}
\label{not all preperiodic}
0<\max_{i=1}^g |x_i-\alpha_i| <R.
\end{equation} 
Assume there exists an infinite sequence of non-negative integers
  $\{n_k\}_{k \ge 0}$ such that
  $\cP^{n_k}(x_1,\dots,x_g)\in V$. We will show that $V$
  must contain a positive dimensional subvariety $V_0$ such that
  $\cP(V_0) = V_0$.
  By Fact~\ref{dynamics}, we have
\begin{equation}
\label{it goes to 0}
P_i^{n_k}(x_i)\to \alpha_i\text{ for each $i\in\{1,\dots,g\}$.}
\end{equation}
After replacing $\{n_k\}_k$ by a subsequence and $R$ by a smaller
positive number, we may assume that $\log_{P_i}$ is well-defined at
$P^{n_k}_i(x_i)$ for each $i\in\{1,\dots,g\}$ and $k\ge 0$, by
\eqref{it goes to 0} and Proposition~\ref{extension-trick}. Similarly,
after shrinking $R$ further, we may also assume that the analytic maps
$\log_{P_i}$ and $\exp_{P_i}$ (defined as in
Proposition~\ref{extension-trick}) are isometries (see also
Proposition~\ref{non-arch anal}).

Because not all $x_i=\alpha_i$, we may assume that
\begin{equation}
\label{maximal norm}
\left|\log_{P_1}\left(P_1^{n_0}(x_1)\right)\right|  = \max_{i=1}^g \left|\log_{P_i}\left(P_i^{n_0}(x_i)\right)\right|  >0,
\end{equation}
using the fact that $\log_{P_i}$ is an isometry between
$B(\alpha_i,R)$ and $B(0;R)$.

We will need to use the following claim.
\begin{claim}
\label{same logarithms}
For each $i\in\{2,\dots, g\}$, the fraction
$\lambda_i:=\frac{\log_{P_i} \left(P^{n_k}_i(x_i)\right)}
{\log_{P_1}\left(P^{n_k}_1(x_1)\right)}$ is independent of $k\ge 0$.
\end{claim}

\begin{proof}[Proof of Claim~\ref{same logarithms}.]
  First of all, it follows immediately from \eqref{maximal norm} and Fact~\ref{dynamics} that
  the denominator of $\lambda_i$ is not zero.  Then, using
  \eqref{fundamental relation 3}, we see that
\begin{equation}
\log_{P_1}(P_1^{n_k}(x_1)) = a_1^{n_k-n_0} \cdot\log_{P_1}\left(P_1^{n_0}(x_1)\right)
\end{equation}
Similarly, we obtain $\log_{P_i}(P_i^{n_k}(x_i)) = a_1^{n_k-n_0}
\cdot\log_{P_i}\left(P_i^{n_0}(x_i)\right)$. This proves
Claim~\ref{same logarithms}.
\end{proof}

For each polynomial $f\in \bC_p[X_1,\dots,X_g]$ in the vanishing ideal of $V$, we construct the following power series
$$F(u):= f\left(u,\exp_{P_2}\left(\lambda_2\cdot \log_{P_1}(u)\right),\dots,\exp_{P_g}\left(\lambda_g\cdot \log_{P_1}(u)\right)\right).$$
Using \eqref{maximal norm} and Claim~\ref{same logarithms}, we conclude that $|\lambda_i| \le 1$ for each $i\in\{2,\dots,g\}$. Therefore, for each $i$ we have that $\lambda_i\cdot \log_{P_1}(u)\in B(0;R)$ if $u\in B(\alpha_1;R)$. Hence $F$ is analytic on $B(\alpha_1;R)$ (we also use the fact that $f$ is a polynomial).

Using Claim~\ref{same logarithms}, we conclude that for each $k\ge 0$, we have
\begin{equation}
\lambda_i\cdot\log_{P_1}\left(P_1^{n_k}(x_1)\right) = \log_{P_i}\left(P_i^{n_k}(x_i)\right)\text{ and so,}
\end{equation}
$$F(P_1^{n_k}(x_1)) =
f\left(P_1^{n_k}(x_1),P_2^{n_k}(x_2),\dots,P_g^{n_k}(x_g)\right) =
0,$$
where in the last equality we used that
$\left(P_1^{n_k}(x_1),\dots,P_g^{n_k}(x_g)\right)\in V$. But
$\lim_{k\to\infty} P_1^{n_k}(x_1)=\alpha_1$ (see \eqref{it goes to
  0}). Because the zeros of $F$ have an accumulation point inside its
domain of convergence, we conclude that $F=0$ (see Lemma~\ref{zeros}). Thus, we have
$$f\left(u,\exp_{P_2}\left(\lambda_2\log_{P_1}(u)\right),\dots,\exp_{P_g}\left(\lambda_g
    \log_{P_g}(u)\right)\right)=0$$
for all polynomials $f$ in the
vanishing ideal of $V$, and for each $u\in B(\alpha_1;R)$.  In particular, each $f$ vanishes on the set
$$\cS = \left\{ \left(P_1^n(x_1),\dots,P_g^n(x_g)\right) \text{ : } n\ge n_0
\right\}$$
Let $W_0$ be the Zariski closure of $\cS$ and let $V_0$ be
the union of the positive dimensional components of $W_0$ (note that
$V_0$ is nonempty since the set $\cS$ is infinite).  Then $V_0$ is a
subvariety of $V$, because any polynomial $f$ vanishes on $V_0$
whenever $f$ vanishes on $V$.  Since $\cP(\cS)$ and $\cS$ differ by at
most one point (namely, the point
$\cP^{n_0}(x_1,\dots,x_g)$), we see that
$\cP(V_0) = V_0$.

Thus, if $V$ contains infinitely many points in the
$\cP$-orbit of $(x_1,\dots,x_g)$, then $V$ must contain a
$\cP$-invariant subvariety of positive dimension.
\end{proof}

The following result is a corollary of our Theorem~\ref{polydyn}
\begin{cor}
\label{p-adic integers}
Let $g\ge 2$, and for each $i\in\{1,\dots,g\}$ we let
$$P_i(X) = a_1X + \sum_{j=2}^{d_i} a_{i,j}X^j\in\bC_p[X].$$
Assume that $0<|a_1| <1$ and that for each $i\in\{1,\dots,g\}$, and
for each $j\ge 2$, we have $|a_{i,j}| \le 1$.

Let $V\subset \bA^g$ be an affine subvariety defined over $\bC_p $
which contains no positive dimensional subvariety invariant under the
$\cP$-action. Then for each $g$-tuple of $p$-adic integers
$(\alpha_1,\dots,\alpha_g)\in \bC_p$ the $\cP$-orbit of
$(\alpha_1,\dots,\alpha_g)$ intersects $V(\bC_p)$ in at most finitely
many points.
\end{cor}

The proof of Corollary~\ref{p-adic integers} is immediate after we
notice that our hypothesis on the coefficients of the polynomials
$P_i$ and on the $\alpha_i$ guarantee us that for each
$i\in\{1,\dots,g\}$, we have $\lim_{n\to \infty} P_i^n(\alpha_i) = 0.$
Furthermore, $0$ is an attracting fixed point for each $P_i$;
therefore, Theorem~\ref{polydyn} shows that $V$ intersects the
$\cP$-orbit of $(\alpha_1,\dots,\alpha_g)$ in at most finitely many
points.

\def\cprime{$'$} \def\cprime{$'$} \def\cprime{$'$} \def\cprime{$'$}
\providecommand{\bysame}{\leavevmode\hbox to3em{\hrulefill}\thinspace}
\providecommand{\MR}{\relax\ifhmode\unskip\space\fi MR }
\providecommand{\MRhref}[2]{%
  \href{http://www.ams.org/mathscinet-getitem?mr=#1}{#2}
}
\providecommand{\href}[2]{#2}

\end{document}